\documentclass[10pt,reqno]{amsart}

\pdfoutput=1
\usepackage{amsmath}
\usepackage{amsfonts}
\usepackage{amsthm}
\usepackage{amssymb}
\usepackage{hyperref}
\usepackage{graphicx}
\usepackage{tikz-cd}
\usepackage{multicol}
\usepackage{enumitem}
\usepackage{indentfirst}
\usepackage[none]{hyphenat}
\usepackage{xcolor}
\usepackage{mathrsfs}
\usepackage{cleveref}
\usepackage{cancel}
\usepackage[foot]{amsaddr}
\usepackage[top=4cm, bottom=4cm, left=2.8cm, right=2.8cm]{geometry}
\usepackage{lmodern}

\newcommand{\sm}{\setminus}
\newcommand{\conn}{\text{co}}
\newcommand{\PP}{\mathbb{P}}

\newcommand{\Ff}{\mathscr F}
\newcommand{\Pp}{\mathscr P}

\newtheorem{theorem}{Theorem}
\newtheorem{lemma}[theorem]{Lemma}
\newtheorem{corollary}{Corollary}

\newtheorem{conjecture}{Conjecture}

\newcommand\numberthis{\addtocounter{equation}{1}\tag{\theequation}}

\newenvironment{nouppercase}{

	\renewcommand{\uppercasenonmath}[1]{}}{}

\makeatletter
\let\@mkboth\@gobbletwo
\def\@oddhead{{\rightmark}\hfil\thepage}
\def\@evenhead{\thepage\hfil{\rightmark}}
\makeatother

\allowdisplaybreaks

\title{{\Large Vertex gluing preserves the bunkbed conjecture}}

\author{\normalsize Paul Meunier$^1$}
\address{$^1$ Department of Mathematics: Analysis, Logic and Discrete Mathematics, Ghent University, Krijgslaan 281, 9000 Gent, Belgium (formerly KU Leuven, Department of Mathematics, Celestijnenlaan 200B box 2400, BE-3001 Leuven), research supported by the grant 11PAL24N funded by the Research Foundation Flanders (FWO)}
\author{\normalsize Pegah Pournajafi$^2$}
\address{$^2$Chaire Combinatoire, Collège de France, Université PSL, 75005, Paris, France}

\begin{document}
	
	\begin{abstract}
		We prove that the bunkbed conjecture is preserved under gluing along a vertex. As a consequence, every minimal counterexample is 2-connected. More generally, for any class of graphs closed under taking 2-connected components, the study of the bunkbed conjecture reduces to the case of 2-connected graphs in that class. In particular, the conjecture holds for forests and block graphs.
	\end{abstract}
	
	\begin{nouppercase}
		\maketitle
	\end{nouppercase}
	
	\section{Introduction}
	Percolation theory has the charming capacity to give rise to a variety of simple-looking problems that remain open for decades. 
	One striking example is the bunkbed conjecture, a folklore conjecture dating back to the 1980s, for which a counterexample was recently announced in~\cite{GladPakZim2024}. 
	This motivates the study of structural properties of bunkbed graphs and of operations under which related inequalities are preserved.	
	
	\smallskip
	
	The bunkbed conjecture concerns connectivity properties on so-called ``bunkbed graphs'', that is, graphs of the form $G = H \Box K_2$, where $\Box$ denotes the box product of graphs and $K_2$ is the complete graph on two vertices. 
	Informally, it states that after an edge percolation process on $G$, the probability that two vertices are connected by a path is higher when they lie in the same layer than when they lie in different layers (see~\Cref{sec:preli_notations} for a formal statement).

	Similar inequalities on bunkbed graphs are known in other probabilistic models, including electrical resistances and discrete random walks~\cite{BollBright1997}, continuous random walks~\cite{Haggstrom1998}, as well as the Ising and random cluster models~\cite{Haggstrom2003}. 
	In the percolation setting, the general conjecture is false, and only a limited number of positive results are known. 
	These include extremal cases where the percolation parameter is sufficiently close to 0 or 1~\cite{Hollom2024, HutchKentNizi2023}, or when the graph is complete~\cite{Buyer2016, Buyer2018, HintLamm2019} or close to being complete (complete bipartite or complete minus one edge)~\cite{Rich2022}.

	Our main result is the following.
	
	\begin{theorem}\label{intro_main_result}
		If the general (resp.\ semihomogeneous) bunkbed conjecture holds for two graphs, then it also holds for any graph obtained by gluing them along a vertex.
	\end{theorem}
	
	As immediate consequences, any minimal counterexample is 2-connected. 
	More generally, for any class of graphs stable under taking 2-connected components, it suffices to verify the conjecture on its 2-connected graphs. 
	In particular, the general bunkbed conjecture holds for forests, and the semihomogeneous version holds for block graphs.

	Since the first version of this work, related results have appeared in the literature, including a new proof of the conjecture for forests~\cite{Donderwinkel2025} as well as a proof for cactus graphs~\cite{Denart2025}.

	\section{Preliminaries and notation}\label{sec:preli_notations}
	
	Several formulations of the bunkbed conjecture appear in the literature. 
	We briefly recall the three main variants before introducing the notation used in this article. 
	We refer to them, in decreasing order of generality, as the general, semihomogeneous, and homogeneous bunkbed conjectures (see their definitions under the statement of~\Cref{cjt:bunkbed}).

	The earliest formulation of the ``bunkbed conjecture'' appears to be in an article by Häggström~\cite{Haggstrom2003}, which is an English version of his earlier Swedish paper~\cite{Haggstrom2002}. 
	Without the name ``bunkbed'', this question already appears as Question 3.1 in~\cite{Haggstrom1998} by the same author, where it is described as part of the probabilistic folklore. 
	Van den Berg and Kahn~\cite{BergKahn2001} also consider a related statement, attributed to discussions with Kasteleyn around 1985. 
	This corresponds to the general version, while Häggström’s formulation corresponds to the homogeneous version.

	Since then, other versions of the conjecture have been studied. 
	In~\cite{RudzSmyth2016}, relationships between different formulations are established. 
	In particular, the authors show that the semihomogeneous version implies the general one, but do not know whether the homogeneous version implies any other. 
	However, this implication does not hold at the level of individual graphs: indeed, the general conjecture is false for complete graphs with sufficiently many vertices, whereas the semihomogeneous version is true for complete graphs~\cite{HintLamm2019}. 
	Since block graphs contain complete graphs, this provides a new class of graphs for which the semihomogeneous version holds but not the general one (see~\Cref{cor:3}).
	
	\medskip
	
	All graphs in this article are finite and simple. We denote the vertex set and the edge set of a graph $G$ by $V(G)$ and $E(G)$ respectively. We often write $xy$ for an edge $\{x,y\} \in E(G)$. 
	
	Let us recall the edge-percolation model.   
	Let $G$ be a graph and let $\mu \colon E(G) \to [0,1]$ be a function. We refer to $ \mu $ as a \emph{weight} on $G$. 
	Define the probability space $(\Omega,\Ff,\PP)$ with respect to $G$ and $\mu$ by setting $\Omega = \Pp(E(G))$, $\Ff = \Pp(\Omega)$ (where $\Pp$ indicates the power set), and defining $\PP \colon \Ff \to [0,1]$ to be the probability measure given on atoms by $\PP(X) = \left(\prod_{e\in X} \mu(e)\right)\left(\prod_{e\notin X} (1-\mu(e))\right)$.
	This corresponds to associating to every edge of the graph a boolean random variable, representing either a closed or open state, where the edge $e$ is open with probability $\mu(e)$ and all edges are independent from each other. 
	
	Given two vertices $x, y\in V(G)$, we define the event $(x\sim_G y)$, or $(x\sim y)$ when it is clear to which graph we are referring to, to be the event consisting of all $K\subseteq E(G)$ for which $x$ and $y$ are in the same connected component of $(V(G),K)$, or equivalently, where there is a path from $x$ to $y$ in $K$. 
	
	In our context, to a pair $(G,\mu)$ consisting of a graph and a weight on it, we always implicitly associate the probability space $(\Omega, \Ff, \PP)$ defined above.  
	We often denote by $\PP_{G,\mu}$ the probability measure $\PP$ to emphasise on the choice of $G$ and $\mu$, especially when working with the same graph but different weights.
	
	The \emph{bunkbed graph} of $ G $, denoted by $ BB(G) $, is the graph $ G \Box K_2 $ (where $ \Box $ denotes the box product of graphs). In other words, we set
	$V(BB(G)) = V(G)\times \{0,1\}$ and 
	$$E(BB(G)) = \{ (x,i)(y,i) \mid xy \in E(G),\ i\in \{0,1\} \} \cup \{(x,0)(x,1) \mid x\in V(G) \}.$$ 
	Hence, $BB(G)$ consists of two copies of $G$, where one adds an edge between the two copies of each vertex. 
	In the rest of the article, for $x\in V(G)$, we denote its two copies $(x,0)$ and $(x,1)$ by $x^-$ and $x^+$ respectively--see \Cref{fig:bunkbed-operation}. A weight $\mu \colon E(BB(G)) \to [0,1]$ on $BB(G)$ is said to be \textit{symmetric} if for every edge $xy \in E(G)$, we have $\mu(x^-y^-) = \mu(x^+y^+)$.

	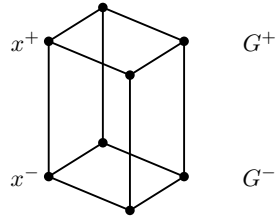
\begin{figure}
		\centering
		\begin{tikzpicture}
			\scalebox{.9}{
				\draw[thick] (.2, 0) -- (1, .5) -- (-.2, 1) -- (-1, .5) -- cycle;
				
				\draw[thick] (0.2, 2) -- (1, 2.5) -- (-0.2, 3) -- (-1, 2.5) -- cycle;
				
				\draw[thick] (0.2, 0) -- (0.2, 2);  
				\draw[thick] (1, .5) -- (1, 2.5);  
				\draw[thick] (-.2, 1) -- (-.2, 3);  
				\draw[thick] (-1, .5) -- (-1, 2.5); 
				
				\fill (0.2, 0) circle (2pt);
				\fill (1, .5) circle (2pt);
				\fill (-0.2, 1) circle (2pt);
				\fill (-1, .5) circle (2pt);
				\fill (0.2, 2) circle (2pt);
				\fill (1, 2.5) circle (2pt);
				\fill (-0.2, 3) circle (2pt);
				\fill (-1, 2.5) circle (2pt);
				
				\node at (-1, .5) [left] {$ x^- $}; 
				\node at (-1, 2.5) [left] {$ x^+ $};  
				\node at (2.5, .5) [left] {$ G^- $}; 
				\node at (2.5, 2.5) [left] {$ G^+ $}; 
			}
		\end{tikzpicture}
		\caption{The bunkbed graph of $ G = C_4 $.} \label{fig:bunkbed-operation}
	\end{figure}

	\medskip 
	Let us now state the bunkbed conjecture. 
	As mentioned earlier, there are several versions of the conjecture, and \Cref{cjt:bunkbed} is the most general one. 
	
	\begin{conjecture}[Bunkbed conjecture]\label{cjt:bunkbed}
		Let $G$ be a graph and let $\mu \colon E(BB(G)) \to [0,1]$ be a symmetric weight on $BB(G) $. 
		Then for every $x$ and $y \in V(G)$, we have
		$\PP_{BB(G),\mu}(x^- \sim y^-) \geq \PP_{BB(G),\mu}(x^-\sim y^+)$.
	\end{conjecture}
	
	Given a pair $(G,\mu)$ where $G$ is a graph and $\mu \colon E(BB(G)) \to [0,1]$ is a symmetric wheight, we say that the pair \emph{satisfies the bunkbed conjecture} if the statement above holds. 
	If it holds for every symmetric wheight $\mu \colon E(BB(G)) \to [0,1]$, we say that the graph $G$ \emph{satisfies the (general) bunkbed conjecture}. 
	We say that $G$ satisfies the \emph{semihomogeneous bunkbed conjecture} if the statement above holds for all symmetric wheights which are constant on horizontal edges. 
	Finally we say that $G$ satisfies the \emph{homogeneous bunkbed conjecture} if the statement above holds for all constant wheights. 
	In the literature it is common to require the vertical edges to have a wheight of 0 or 1, which corresponds to conditioning on whether these edges are present or not in the percolation process. 
	In~\cite[Theorem 6.1]{GladPakZim2024}, the authors introduce a counterexample to the homogeneous version, which also refutes the semihomogeneous and general versions.

	For sets, we write $C = A \sqcup B$ to show that the intersection of $A$ and $B$ is empty (notice that this does not mean that we have a partition of $ C $ as $ A $ and $ B $ are allowed to be the empty set). For a function $ f $ and a subset $ A $ of its domain, we denote the restriction of $ f $ to $ A $ by $ f|_A$. We may simply write $ f $ instead of $ f|_A $ if it is clear from the context which restriction is set on the domain. For a graph $ G $ and a set $ X \subseteq V(G) $, we denote by $ G[X] $ the subgraph of $ G $ \emph{induced} by $ X $--that is, the subgraph whose vertex set is $ X $ and whose edge set is the set of edges with both end-points in $ X $. 
	
	Let $ K \subseteq E(G) $ and let $ x, y \in V(G) $. We write $\conn_{K}(x,y)$ for a fixed logical formula stating that  there exists a path between $ x $ and $ y $ consisting only of edges of $ K $. Thus, $ \neg \conn_{K}(x,y)$ states that there exists no path between $x$ and $y$ using only edges of $ K$. 
	
	A \emph{cut vertex} in $ G $ is a vertex $ v \in V(G) $ such that the number of connected components of $G \setminus \{v\} $ is strictly bigger than the number of connected components of $ G $. Let $ G_1 $ and $ G_2 $ be two graphs and let $ v_i \in V(G_i) $. \emph{Gluing} $G_1$ and $G_2$ along $v_1 $ and $v_2$ is to obtain a graph by first taking the union of $ G_1 $ and $ G_2 $ and then identifying $ v_1 $ and $v_2$. Notice that the vertex obtained from this identification is a cut vertex if $ v_i $ is not an isolated vertex in $G_i$ for $i\in\{1,2\}$.
	
	\section{Gluing along a vertex}

	Let us fix the notation for the rest of the article.
	
	Let $\overline F$ be a graph and let $v \in V(\overline F)$ be a cut vertex. 
	Let $V_1, \dots, V_k$ (where $k \geq 2$) be the vertex sets of the connected components of $\overline F \sm \{v\}$. 
	Fix $ I \subseteq \{1, \dots, k\} $. Set $\overline G = \overline F[\cup_{i \in I} V_i \cup \{v\}]$ and $\overline H = \overline F[\cup_{i \notin I} V_i \cup \{v\}]$. 
	Notice that $ \overline F $ is the graph obtained by gluing $ \overline G $ and $ \overline H $ along a vertex (namely, $v$).
	Now set $F = BB(\overline F)$, $G = BB(\overline G)$, and $H = BB(\overline H)$. 
	Finally, let $H_0 = H\sm \{v^-v^+\}$. 
	See \Cref{fig:fixed-notation}. Notice that every path from a vertex $ a \in G $ to a vertex $ b \in H $ goes through either $ v^- $ or $ v^+$. 
	
	\begin{figure}
		\centering
		\begin{tikzpicture}
			\scalebox{.9}{
				\draw[double, double distance=4pt] (-1.5, 1.5) -- (-1.5, 0.5);
				\draw[double, double distance=4pt] (1.5, 1.5) -- (1.5, 0.5);
				\draw (0,0) -- (0,2);

				\draw[thick, black] (-1.5, 0) ellipse (2cm and .5cm); 
				\draw[thick, black] (1.5, 0) ellipse (2cm and .5cm); 
				
				\node at (-1.5, 0) {$ G^- $};
				\node at (1.5, 0) {$ H^- $};
				
				\fill (0.0, 0) circle (1pt);
				\node at (-.08, -.1) [right] {$ v^- $};
				
				\node at (-4.5, 0) {$ F^- $};

				\draw[thick, black] (-1.5, 2) ellipse (2cm and .5cm); 
				\draw[thick, black] (1.5, 2) ellipse (2cm and .5cm); 
				
				\node at (-1.5, 2) {$ G^+ $};
				\node at (1.5, 2) {$ H^+ $};
				
				\fill (0.0, 2) circle (1pt);
				\node at (-.08, 1.9) [right] {$ v^+ $};
				
				\node at (-4.5, 2) {$ F^+ $};
			}		
		\end{tikzpicture}
		\caption{The notation for Section 3.} \label{fig:fixed-notation}
	\end{figure}
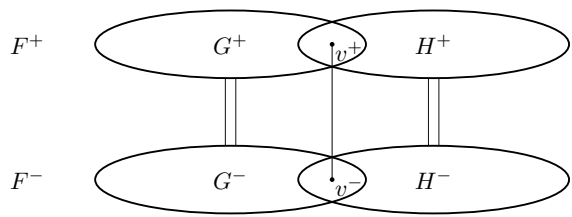
	
	In the rest of the text, whenever we write $ G$, $H$, $H_0 $, etc.\ we refer to the notation above.
	
	\medskip
	
	Our goal is to prove \Cref{thm:main-vertex-cut} together with Corollaries \ref{cor:1} and \ref{cor:2}. The idea behind \Cref{thm:main-vertex-cut} is that if the bunkbed conjecture is true for $ \overline{G} $ and $ \overline{H} $, then it is true for $ \overline F $ as well. In order to prove that
	$\PP(x^- \sim y^+) \leq \PP(x^- \sim y^-)$, for every $ x, y \in V(F) $, we distinguish two cases. First, in \Cref{lem:mainthm-case1}, we obtain the inequality when $ x $ and $ y $ are on the same `side' of $ v $, i.e.\ when they are both in $ V(\overline H) $ for instance. Then, in \Cref{lem:mainthm-case2}, we treat the case where $ x $ and $ y $ are on two different `sides' of $ v $, i.e.\ when one is in $V(\overline  G)$ and the other in $V(\overline  H)$.

	\subsection{First case}
	
	In this subsection, we study the case where $ x $ and $ y $ are both vertices of $ \overline H $. The idea is that the situation then becomes equivalent to working solely in $ H$ but with a different weight.
	
	\begin{lemma}\label{lem:change_weight}
		Let $\mu$ be a weight on $F$. 
		Define the following weight $\mu'$ on $H$: for $e \in E(H)$, set
		$$
		\mu' (e) = \begin{cases}
			\mu(e) & e \neq v^-v^+ \\
			\PP_{G, \mu|_G}(v^- \sim v^+) & e = v^-v^+
		\end{cases}.
		$$ \enlargethispage{\baselineskip} 
		Then, for every $x, y \in V(H)$, we have 
		$ \PP_{F, \mu}(x \sim_F y) = \PP_{H, \mu'}(x \sim_H y)$.
	\end{lemma}
	
	\begin{proof} 
		Recall that $H_0 = H \sm \{v^-v^+\}$. 
		Notice that $E(F) = E(H_0) \sqcup E(G)$.  
		Let $ A $ be the event $ (x \sim_F y) $ and set 
		\begin{align*}
			A_0 &= \{ K \subseteq E(F) \mid \neg\conn_{K\cap E(H_0)}(x,y), \conn_{K}(x,y) \},\\
			\text{and } A_1 &= \{ K \subseteq E(F) \mid \conn_{K\cap E(H_0)}(x,y) \}.
		\end{align*}
		Notice that $A = A_0 \sqcup A_1$. 
		
		Now notice that if for a given $K \subseteq E(F)$, we have $\neg\conn_{K\cap E(H_0)}(x,y)$ and $\conn_K(x,y)$, then every path between $x$ and $y$ in $K$ uses edges from $G$. Therefore, we have $A_0 = A_2 \cup A_3$, where  
		\begin{align*}
			A_2 &= \{ K \subseteq E(F) \mid \neg\conn_{K\cap E(H_0)}(x,y), \conn_{K\cap E(H_0)}(x,v^-), \conn_{K\cap E(H_0)}(y,v^+), \conn_{K\cap E(G)}(v^-,v^+) \} \\
			\text{and }
			A_3 & = \{  K \subseteq E(F) \mid \neg\conn_{K\cap E(H_0)}(x,y), \conn_{K\cap E(H_0)}(x,v^+), \conn_{K\cap E(H_0)}(y,v^-), \conn_{K\cap E(G)}(v^-,v^+) \}.
		\end{align*}
		Moreover, if there exists $ K \in  A_2 \cap A_3 $, then we have $ \conn_{K\cap E(H_0)}(x,v^+) $ and $\conn_{K\cap E(H_0)}(y,v^+)$ which would imply that we have $ \conn_{K\cap E(H_0)}(x,y) $, contradicting the condition  $\neg\conn_{K\cap E(H_0)}(x,y) $. Hence,  $A_2 \cap A_3 = \varnothing $.
		Therefore, we obtain $A = A_1 \sqcup A_2 \sqcup A_3$, and, as a result, we have 
		\begin{align*}
			\PP_{F, \mu}(A) = \PP_{F, \mu}(A_1) + \PP_{F, \mu}(A_2) + \PP_{F, \mu}(A_3). \numberthis \label{eq:Asplit}
		\end{align*}
		Now, notice that $E(F) = E(H_0)\sqcup E(G)$. 
		We have:
		\begin{align*}
			\PP_{F, \mu}(A_1) = \sum_{K \in A_1} \PP_{F,\mu}(K) &= \sum_{\substack{K_1 \subseteq E(H_0) \\ K_2 \subseteq E(G) \\ \conn_{K_1}(x,y)}} \PP_{F,\mu}(K_1 \sqcup K_2) \\
			&= \sum_{\substack{K_1 \subseteq E(H_0) \\ K_2 \subseteq E(G) \\ \conn_{K_1}(x,y)}} \PP_{H_0,\mu}(K_1) \PP_{G, \mu}(K_2)   \\ 
			&= \left(\sum_{\substack{K \subseteq E(H_0) \\ \conn_{K}(x,y)}} \PP_{H_0,\mu}(K)\right) \left(\sum_{K \subseteq E(G)} \PP_{G, \mu}(K) \right)\\
			&= \PP_{H_0, \mu}(x \sim_{H_0} y), \numberthis \label{eq:A1_calc}
		\end{align*}
		where the last line follows from the definition of $\PP_{H_0, \mu_0}(x \sim_{H_0} y)$ and the fact that $\sum_{K \subseteq E(G)} \PP_{G, \mu}(K) = 1$.  
		Moreover, we have:
		\begin{align*}	
			\PP_{F, \mu}(A_2) = \sum_{K \in A_2} \PP_{F, \mu}(K) &= \sum_{\substack{K_1 \subseteq E(H_0), K_2 \subseteq E(G) \\ \neg\conn_{K_1}(x,y), \conn_{K_1}(x,v^-), \conn_{K_1}(y,v^+) \\ \conn_{K_2}(v^-,v^+)}} \PP_{F, \mu}(K_1 \sqcup K_2) \\
			& = \sum_{\substack{K_1 \subseteq E(H_0), K_2 \subseteq E(G) \\ \neg\conn_{K_1}(x,y), \conn_{K_1}(x,v^-), \conn_{K_1}(y,v^+) \\ \conn_{K_2}(v^-,v^+)}} \PP_{H_0, \mu}(K_1) \PP_{G, \mu}(K_2)   \\
			&= \left(\sum_{\substack{K \subseteq E(H_0) \\ \neg\conn_{K}(x,y), \conn_{K}(x,v^-), \conn_{K}(y,v^+)}} \PP_{H_0, \mu}(K)\right) \left(\sum_{\substack{K \subseteq E(G) \\ \conn_{K}(v^-,v^+)}} \PP_{G, \mu}(K) \right)\\
			& = \left( \sum_{\substack{K \subseteq E(H_0) \\ \neg\conn_{K}(x,y), \conn_{K}(x,v^-), \conn_{K}(y,v^+)}} \PP_{H_0, \mu}(K) \right)  \PP_{G, \mu}(v^- \sim_G v^+).  \numberthis \label{eq:A2_calc}
		\end{align*}
		Similarly, we have:
		\begin{align*}	
			\PP_{F, \mu}(A_3) = \sum_{K \in A_3} \PP_{F, \mu}(K) = \left( \sum_{\substack{K \subseteq E(H_0) \\ \neg\conn_{K}(x,y), \conn_{K}(x,v^+), \conn_{K}(y,v^-)}} \PP_{H_0, \mu}(K) \right) \PP_{G, \mu}(v^- \sim_G v^+). \numberthis \label{eq:A3_calc}
		\end{align*}
		Now, let $B $ be the event $ (x \sim_H y) $. 
		Setting
		\begin{align*}		
			& B_0 = \{ K \subseteq E(H) \mid \neg\conn_{K\cap E(H_0)}(x,y), \conn_K(x,y) \}
			\\
			\text{and } & B_1 = \{ K \subseteq E(H) \mid \conn_{K\cap E(H_0)}(x,y) \},
		\end{align*}
		we have $B = B_0 \sqcup B_1$. 	
		Moreover, similar to before, we have $B_0 = B_2 \sqcup B_3$, where 
		\begin{align*}
			& B_2 = \{ K \subseteq E(H) \mid \neg\conn_{K\cap E(H_0)}(x,y), \conn_{K\cap E(H_0)}(x,v^-), \conn_{K\cap E(H_0)}(y,v^+), v^- v^+\in K \},
			\\
			\text{and } & B_3 =   \{ K \subseteq E(H) \mid \neg\conn_{K\cap E(H_0)}(x,y), \conn_{K\cap E(H_0)}(x,v^+), \conn_{K\cap E(H_0)}(y,v^-), v^- v^+ \in K \}.
		\end{align*}
		Therefore, $B = B_1 \sqcup B_2 \sqcup B_3$, and 
		\begin{align*}
			\PP_{H, \mu'}(B)  = \PP_{H, \mu'}(B_1) + \PP_{H, \mu'}(B_2) + \PP_{H, \mu'}(B3). \numberthis \label{eq:Bsplit}
		\end{align*}
		Notice that:
		\begin{align*}
			\PP_{H, \mu'}(B_1) = \sum_{K \in B_1} \PP_{H, \mu'}(K) & =  \sum_{\substack{K_1  \subseteq E(H_0) \\ K_2 \subseteq \{v^- v^+\} \\ \conn_{K_1}(x,y)}} \PP_{H, \mu'}(K_1 \sqcup K_2) \\
			& = \sum_{\substack{K_1  \subseteq E(H_0) \\ K_2 \subseteq \{v^-v^+\} \\ \conn_{K_1}(x,y)}} \PP_{H_0, \mu'}(K_1) \PP_{\{v^- v^+\}, \mu'}(K_2)   \\
			& = \left( \sum_{\substack{K  \subseteq E(H_0) \\ \conn_{K}(x,y)}} \PP_{H_0, \mu'}(K)\right) \left(\sum_{K \subseteq \{v^- v^+\}} \PP_{\{v^- v^+\}, \mu'} (K) \right) \\
			& = \left(\sum_{\substack{K  \subseteq E(H_0) \\ \conn_{K}(x,y)}} \PP_{H_0, \mu}(K) \right). 1 \tag{since $ \mu'|_{H_0} = \mu|_H $} \\
			& = \PP_{H_0, \mu} (x \sim_{H_0} y). \numberthis \label{B1_calc}
		\end{align*}
		Moreover, we have:
		\begin{align*}
			\PP_{H, \mu'}(B_2) = \sum_{K \in B_2} \PP_{H, \mu'}(K) &= \sum_{\substack{K_1 \subseteq E(H_0) \\ \neg\conn_{K_1}(x,y), \conn_{K_1}(x,v^-), \conn_{K_1}(y,v^+) }} \PP_{H, \mu'}(K_1 \sqcup \{v^-v^+\}) \\
			& = \sum_{\substack{K_1 \subseteq E(H_0)\\ \neg\conn_{K_1}(x,y), \conn_{K_1}(x,v^-), \conn_{K_1}(y,v^+) }} \PP_{H_0, \mu'}(K_1) \PP_{\{v^-v^+\}, \mu'}(\{v^-v^+\})  \\
			& = \left( \sum_{\substack{K \subseteq E(H_0) \\ \neg\conn_{K}(x,y), \conn_{K}(x,v^-), \conn_{K}(y,v^+) }} \PP_{H_0, \mu}(K) \right) \mu'(v^-v^+) \tag{using $ \mu'|_{H_0} = \mu|_H $} \\
			& = \left( \sum_{\substack{K \subseteq E(H_0) \\ \neg\conn_{K}(x,y), \conn_{K}(x,v^-), \conn_{K}(y,v^+)}} \PP_{H_0, \mu}(K) \right)  \PP_{G, \mu}(v^- \sim_G v^+),  \numberthis \label{eq:B2_calc}
		\end{align*}
		where in the last line we have used the definition of $\mu'$.	
		
		Similarly, we obtain:
		\begin{align*}
			\PP_{H, \mu'}(B_3) =  \sum_{K \in B_3} \PP_{H, \mu'}(K) = 
			\left( \sum_{\substack{K \subseteq E(H_0) \\ \neg\conn_{K}(x,y), \conn_{K}(x,v^+), \conn_{K}(y,v^-)}} \PP_{H_0, \mu}(K) \right)  \PP_{G, \mu}(v^- \sim_G v^+).  \numberthis \label{eq:B3_calc}
		\end{align*}	
		
		Now, by \Cref{eq:A1_calc} and \Cref{B1_calc}, we have  $ \PP_{F, \mu}(A_1) =  \PP_{H, \mu'}(B_1) $, by \Cref{eq:A2_calc} and \Cref{eq:B2_calc}, we have $ \PP_{F, \mu}(A_2) = \PP_{H, \mu'}(B_2) $, and finally by \Cref{eq:A3_calc} and \Cref{eq:B3_calc}, we have $ \PP_{F, \mu}(A_3) = \PP_{H, \mu'}(B_3) $. Therefore, using \Cref{eq:Asplit} and \Cref{eq:Bsplit}, we deduce
		$\PP_{F, \mu}(x \sim_F y) =  \PP_{H, \mu'}(x \sim_H y)$, which completes the proof.
	\end{proof}

	\begin{lemma} \label{lem:mainthm-case1}
		Let $\mu$ be a symmetric wheight on $F$. 
		Assume that $(\overline H, \nu)$ satisfies the bunkbed conjecture for all symmetric wheights $\nu$ on $F$ that coincide with $\mu$ on horizontal edges. 
		Then for every $x, y \in V(\overline H)$, we have $\PP_{F, \mu}(x^- \sim_F y^+) \leq \PP_{F, \mu}(x^- \sim_F y^-)$. 
	\end{lemma}
	
	\begin{proof}
		By Lemma~\ref{lem:change_weight} and by assumption on $H$, since $\mu'$ coincides with $\mu$ on horizontal edges (actually it coincides with $\mu$ on all edges but one vertical edge), we have that
		\[ \PP_{F,\mu}(x^- \sim_F y^+) = \PP_{H,\mu'}(x^-\sim_H y^+) \leq \PP_{H,\mu'}(x^-\sim_H y^-) = \PP_{F,\mu}(x^-\sim_F y^-),\]
		as desired.
	\end{proof}
	
	\subsection{Second case}
	
	We now study the case when $x$ and $y$ are on two different sides of the cut vertex, that is, $ x \in V(\overline G) $ and $ y \in V(\overline H)$ (up to renaming).
	
	\begin{lemma} \label{lem:breakP(xy)to2sides}
		Let $\mu$ be a weight on $F$. 
		Then, for every $x \in V(G)$ and $y \in V(H)$, we have: 
		\begin{align*}
			\PP_{F, \mu}(x \sim_F y) =\  & \PP_{G, \mu}(x \sim_G v^-)\PP_{H_0, \mu}(v^- \sim_{H_0} y) + \PP_{G, \mu}(x \sim_G v^+)\PP_{H_0, \mu}(v^+ \sim_{H_0} y) \\
			& - \PP_{G, \mu}(x \sim_G v^- \cap x \sim_G v^+)\PP_{H_0, \mu}(v^- \sim_{H_0} y \cap v^+ \sim_{H_0} y).
		\end{align*}
		
	\end{lemma}
	
	\begin{proof}
		Let $A $ be the event $ (x\sim_F y) $. Set 
		\begin{align*}
			& A_1 = \{ K\subseteq E(F) \mid \conn_{K\cap E(G)}(x,v^-), \conn_{K\cap E(H_0)}(v^-,y)\} \\
			\text{and } & A_2 = \{ K\subseteq E(F) \mid \conn_{K\cap E(G)}(x,v^+), \conn_{K\cap E(H_0)}(v^+,y)\}.
		\end{align*}
		It is clear that $A_1\cup A_2\subseteq A$, and, since $v$ is a cut vertex in $\overline F$, we obtain that $A = A_1\cup A_2$. 
		Thus by the inclusion-exclusion formula we have $\PP_{F,\mu}(A) = \PP_{F,\mu}(A_1) + \PP_{F,\mu}(A_2) - \PP_{F,\mu}(A_1\cap A_2).$ Let us compute each term.
		
		First, we have:
		\begin{align*}
			\PP_{F,\mu}(A_1) & = \sum_{ \substack{K \subseteq E(F) \\ \conn_{K \cap E(G)}(x,v^-) \\ \conn_{K \cap E(H_0)}(v^-,y) }} \PP_{F, \mu}(K)  \\
			& = \sum_{ \substack{K_1 \subseteq E(G) \\ \conn_{K_1}(x,v^-) \\  K_2 \subseteq E(H_0) \\ \conn_{K \cap E(H_0)}(v^-,y) }} \PP_{F, \mu} (K_1 \sqcup K_2) \tag{setting $ K_1 = K \cap E(G) $ and $ K_2 = K \cap E(H_0) $} \\
			& = \left( \sum_{ \substack{K_1 \subseteq E(G) \\ \conn_{K_1}(x,v^-) }} \PP_{G, \mu}(K_1) \right) \left( \sum_{ \substack{ K_2 \subseteq E(H_0) \\ \conn_{K \cap E(H_0)}(v^-,y) }} \PP_{H_0, \mu}(K_2) \right)  \\
			& = \PP_{G,\mu}(x\sim_G v^-)\PP_{H_0,\mu}(v^-\sim_{H_0} y),
		\end{align*}
		as desired. 
		Similarly, we have that 
		$\PP_{F,\mu}(A_2) = \PP_{G,\mu}(x\sim_G v^+)\PP_{H_0,\mu}(v^+\sim_{H_0} y).$
		
		Finally, we have:
		$$A_1\cap A_2 = \{K\subseteq E(F) \mid \conn_{K\cap E(G)}(x,v^-), \conn_{K\cap E(G)}(x,v^+), \conn_{K\cap E(H_0)}(v^-,y), \conn_{K\cap E(H_0)}(v^+,y)\}.$$ 
		A similar computation leads to $ \PP_{F,\mu}(A_1\cap A_2) = \PP_{G,\mu}(x\sim_G v^- \cap x\sim_G v^+)\PP_{H_0,\mu}(v^-\sim_{H_0} y\cap v^+\sim_{H_0} y) $,
		which concludes the proof.
	\end{proof}
	
	\begin{lemma} \label{lem:mainthm-case2}
		Let $\mu$ be a symmetric weight on $F$. 
		Assume that $(\overline G,\alpha)$ and $(\overline H,\beta)$ satisfy the bunkbed conjecture for all symmetric wheights $\alpha \colon E(G) \to [0,1]$ and $\beta \colon E(H) \to [0,1]$ which coincide with $\mu$ on horizontal edges. 
		Then $\PP_{F, \mu}(x^- \sim_F y^+) \leq \PP_{F, \mu}(x^- \sim_F y^-)$ for all $x\in V(\overline G)$ and $y\in V(\overline H)$. 
	\end{lemma}
	
	\begin{proof}
		Using \Cref{lem:breakP(xy)to2sides} twice, we have:
		\begin{align*}
			\PP_{F, \mu}(x^- \sim_F y^+) = \ & \PP_{G, \mu}(x^- \sim_G v^-)\PP_{H_0, \mu}(v^- \sim_{H_0} y^+) + \PP_{G, \mu}(x^- \sim_G v^+)\PP_{H_0, \mu}(v^+ \sim_{H_0} y^+) \\
			&- \PP_{G, \mu}(x^- \sim_G v^- \cap x^- \sim_G v^+)\PP_{H_0, \mu}(v^- \sim_{H_0} y^+ \cap v^+ \sim_{H_0} y^+),
		\end{align*}
		and 
		\begin{align*}
			\PP_{F, \mu}(x^- \sim_F y^-) = \ & \PP_{G, \mu}(x^- \sim_G v^-)\PP_{H_0, \mu}(v^- \sim_{H_0} y^-) + \PP_{G, \mu}(x^- \sim_G v^+)\PP_{H_0, \mu}(v^+ \sim_{H_0} y^-) \\ &- \PP_{G, \mu}(x^- \sim_G v^- \cap x^- \sim_G v^+)\PP_{H_0, \mu}(v^- \sim_{H_0} y^- \cap v^+ \sim_{H_0} y^-).
		\end{align*}
		Subtracting the two sides of the two equations, we have: 
		\begin{align*}
			& \PP_{F, \mu}(x^- \sim_F y^-) - \PP_{F, \mu}(x^- \sim_F y^+)  \\
			& = \PP_{G, \mu}(x^- \sim_G v^-) \left( \PP_{H_0, \mu}(v^- \sim_{H_0} y^-) - \PP_{H_0, \mu}(v^- \sim_{H_0} y^+) \right) \\
			& + \PP_{G, \mu}(x^- \sim_G v^+) \left( \PP_{H_0, \mu}(v^+ \sim_{H_0} y^-) - \PP_{H_0, \mu}(v^+ \sim_{H_0} y^+) \right) \\
			& - \PP_{G,\mu}(x^- \sim_G v^- \cap x^- \sim_G v^+) \left(\PP_{H_0,\mu}(v^- \sim_{H_0} y^- \cap v^+ \sim_{H_0} y^-) - \PP_{H_0,\mu}(v^-\sim_{H_0} y^+ \cap v^+\sim_{H_0} y^+) \right). \numberthis \label{proof-of-lemma-case2-mid-eq}
		\end{align*}
		Now, since $ \mu $ is symmetric, the last term is 0. 
		Again by symmetry, the two first summands are opposite, since:
		\begin{align*}
			\PP_{H_0, \mu}(v^- \sim_{H_0} y^-) - \PP_{H_0, \mu}(v^- \sim_{H_0} y^+) & = \PP_{H_0, \mu}(v^+ \sim_{H_0} y^+) - \PP_{H_0, \mu}(v^+ \sim_{H_0} y^-)  \\
			& = - \left( \PP_{H_0, \mu}(v^+ \sim_{H_0} y^-) - \PP_{H_0, \mu}(v^+ \sim_{H_0} y^+) \right).
		\end{align*}
		Therefore, \Cref{proof-of-lemma-case2-mid-eq} is equal to: 
		\begin{align*}
			\left( \PP_{H_0, \mu}(v^- \sim_{H_0} y^-) - \PP_{H_0, \mu}(v^- \sim_{H_0} y^+) \right) \left( \PP_{G,\mu}(x^- \sim_G v^-) - \PP_{G,\mu}(x^- \sim_G v^+) \right).
		\end{align*}
		
		By assumption, the second expression is nonnegative. 
		Let us show the first expression is also nonnegative. 
		Let $\mu' \colon E(H) \to [0,1]$ be equal to $\mu$ on $E(H_0)$ and set $\mu'(v^-v^+) = 0$. 
		Notice that $\mu'$ is a symmetric weight on $H$ and that, for $K \subseteq E(H_0)$, we have:
		$$
		\PP_{H,\mu'}(K) =
		\begin{cases}
			0 & v^- v^+ \in K \\
			\PP_{H_0,\mu}(K) & \text{ if } K \subseteq E(H_0) 
		\end{cases}.
		$$
		Therefore, we have:
		\begin{align*}
			\PP_{H, \mu'}(v^- \sim_H y^-) = \sum_{ \substack{ K\subseteq E(H)  \\ \conn_K(v^- , y^-)}} \PP_{H,\mu'} (K)
			&  = \sum_{ \substack{ K\subseteq E(H_0) \\ \conn_K(v^- , y^-) }} \PP_{H, \mu'}(K) + \sum_{ \substack{ K \subseteq E(H),   v^-v^+ \in K \\ \conn_K(v^- , y^-) }} \PP_{H, \mu'}(K) \\
			& = \sum_{ \substack{ K\subseteq E(H_0) \\ \conn_K(v^- , y^-) }} \PP_{H_0, \mu}(K) + 0 \\ 
			& = \PP_{H_0, \mu}(v^- \sim_{H_0} y^-).
		\end{align*}
		Similarly, $ \PP_{H, \mu'}(v^- \sim_H y^+) = \PP_{H_0, \mu}(v^- \sim_{H_0} y^+) $.

		Therefore, 
		$ \PP_{H_0, \mu}(v^- \sim_{H_0} y^-) - \PP_{H_0, \mu}(v^- \sim_{H_0} y^+) = \PP_{H, \mu'}(v^- \sim_{H} y^-) - \PP_{H, \mu'}(v^- \sim_{H} y^+) $, which is nonnegative by assumption since $\mu'$ coincides with $\mu$ on horizontal edges. 
		This shows that \Cref{proof-of-lemma-case2-mid-eq} is nonnegative and concludes the proof.
	\end{proof}
	
	\subsection{Main theorem and corollaries}\label{subsec:main_thm}	
	
	Let us now proceed to state and prove our mains results. 
	
	\begin{theorem} \label{thm:main-vertex-cut}
		Let $A$ and $B$ be two graphs, and fix $a \in V(A)$ and $b\in V(B)$. 
		Let $D$ be the graph obtained from $A$ and $B$ by identifying $a$ and $b$, and let $\mu \colon E(BB(D)) \to [0,1]$ be a symmetric wheight. 
		
		If $(A,\alpha)$ and $(B,\beta)$ satisfy the bunkbed conjecture for every wheights $\alpha \colon E(BB(A)) \to [0,1]$ and $\beta \colon E(BB(B)) \to [0,1]$ which coincide with $\mu$ on horizontal edges, then $(D,\mu)$ satisfies the bunkbed conjecture. 
		
		In particular, if $A$ and $B$ satisfy the general (resp.\ semihomogeneous) bunkbed conjecture, then $D$ satisfies the general (resp.\ semihomogeneous) bunkbed conjecture.
	\end{theorem}
	
	\begin{proof}
		When applying previous lemmas, we take $\overline G = A$, $\overline H = B$, and $\overline F = D$. 
		Also, $v $ corresponds to the vertex obtained after identifying $a$ and $b$. 
		
		Let $\mu \colon E(D) \to [0,1]$ be a symmetric weight and let $x$ and $y\in V(D)$. 
		Up to renaming, two cases are possible: either $x$ and $y$ are both vertices of $B$, or $x$ is a vertex of $A$ and $y$ is a vertex of $B$. 
		In the former case, since $B$ satisfies the general (resp.\ semihomogeneous) bunkbed conjecture, by \Cref{lem:mainthm-case1}, we have that $\PP_{D,\mu}(x^-\sim_D y^+) \leq \PP_{D,\mu}(x^-\sim_D y^+)$, as desired. 
		In the latter case, since $A$ and $B$ satisfy the general (resp.\ semihomogeneous) bunkbed conjecture, by \Cref{lem:mainthm-case2}, we have the desired inequality. 
		This concludes the proof of the theorem.
	\end{proof}
	
	The attentive reader will notice that one could weaken even more the assumptions in the main theorem, since technically speaking the lemmas only need the value of the wheight of the vertical edge above the cut vertex to be modified, and, by conditioning, one could even restrict that value to be 0 or 1. 
	
	We now obtain immediately the following corollaries. 
	
	\begin{corollary} \label{cor:1}
		If $ G$ is a minimal counterexample to the general (resp.\ semihomogeneous) bunkbed conjecture, then it is 2-connected. 
	\end{corollary}
	
	This corollary can be strengthened as follows.
	
	\begin{corollary} \label{cor:2}
		Let $\mathcal C$ be a class of graphs stable under taking 2-connected components. 
		If the general (resp.\ semihomogeneous) bunkbed conjecture holds for the 2-connected graphs in $\mathcal C$, then it holds for every graph in $\mathcal C$. 
	\end{corollary}
	
	\Cref{cor:2} implies the validity of the conjecture in the following cases. 
	Recall that a forest is a graph without cycles, and that a block graph is a graph whose 2-connected components are complete. 
	
	\begin{corollary} \label{cor:3}
		The general bunkbed conjecture is true for forests, and the semihomogeneous bunkbed conjecture is true for block graphs.
	\end{corollary}
	
	\begin{proof}
		By~\Cref{cor:2}, the fact that the general bunkbed conjecture is true for $K_2$ implies that it is true for forests, since $K_2$ is the only 2-connected forest. 
		
		The 2-connected block graphs are exactly complete graphs, and the semihomogeneous bunkbed conjecture is true for complete graphs by the main result of~\cite{HintLamm2019} (the result is stated for the case where vertical edges have wheight 0 or 1, which corresponds to conditioning on which vertical edges are present, so it implies the semihomogeneous case as well). 
		Thus the semihomogeneous bunkbed conjecture is true for block graphs.
	\end{proof}

	\section*{Acknowledgements}
	We warmly thank Marc Becker for introducing the problem that led to this work.


\begin{thebibliography}{10}
		
		\bibitem{BollBright1997}
		B{\'e}la Bollob{\'a}s and Graham Brightwell.
		\newblock Random walks and electrical resistances in products of graphs.
		\newblock {\em Discrete applied mathematics}, 73(1):69--79, 1997.
		
		\bibitem{Buyer2016}
		Paul de~Buyer.
		\newblock A proof of the bunkbed conjecture for the complete graph at $
		p=\frac{1}{2}$.
		\newblock {\em arXiv:1604.08439}, 2016.
		
		\bibitem{Buyer2018}
		Paul de~Buyer.
		\newblock A proof of the bunkbed conjecture on the complete graph for $
		p \geq 1/2$.
		\newblock {\em arXiv:1802.04694}, 2018.
		
		\bibitem{Denart2025}
		Robin Denart.
		\newblock The bunkbed conjecture still holds for cactus graphs and for graphs with certain biconnected components.
		\newblock {\em arXiv:2506.09264}, 2025.
		
		\bibitem{Donderwinkel2025}
		Serte Donderwinkel, Joost Jorritsma, and Guillem Perarnau.
		\newblock An elementary proof of the bunkbed conjecture for forests.
		\newblock {\em arXiv:2511.13589}, 2025.	
		
		\bibitem{GladPakZim2024}
		Nikita Gladkov, Igor Pak, and Aleksandr Zimin.
		\newblock The bunkbed conjecture is false.
		\newblock {\em Proceedings of the National Academy of Sciences}, 122(24):e2420725122, 2025.
		
		\bibitem{Haggstrom1998}
		Olle H{\"a}ggstr{\"o}m.
		\newblock On a conjecture of Bollob{\'a}s and Brightwell concerning random
		walks on product graphs.
		\newblock {\em Combinatorics, Probability and Computing}, 7(4):397--401, 1998.
		
		\bibitem{Haggstrom2002}
		Olle H{\"a}ggstr{\"o}m.
		\newblock Sannolikhetsteori p{\aa} tv{\aa}v{\aa}ningsgrafer.
		\newblock \emph{NORMAT - Nordisk Matematisk Tidskrift} 50:170--180, 2002.
		
		\bibitem{Haggstrom2003}
		Olle H{\"a}ggstr{\"o}m.
		\newblock Probability on bunkbed graphs.
		\newblock In {\em Proceedings of Formal Power Series and Algebraic
			Combinatorics}, 3:19--27, 2003.
		
		\bibitem{Hollom2024}
		Lawrence Hollom.
		\newblock A new proof of the bunkbed conjecture in the $p \uparrow 1$ limit.
		\newblock {\em Discrete Mathematics}, 347(1):113711, 2024.
		
		\bibitem{HutchKentNizi2023}
		Tom Hutchcroft, Alexander Kent, and Petar Nizi{\'c}-Nikolac.
		\newblock The bunkbed conjecture holds in the limit.
		\newblock {\em Combinatorics, Probability and Computing}, 32(3):363--369, 2023.
		
		\bibitem{Rich2022}
		Thomas Richthammer.
		\newblock Bunkbed conjecture for complete bipartite graphs and related classes
		of graphs.
		\newblock {\em arXiv:2204.12931}, 2022.
		
		\bibitem{RudzSmyth2016}
		James Rudzinski and Clifford Smyth.
		\newblock Equivalent formulations of the bunk bed conjecture.
		\newblock {\em The North Carolina Journal of Mathematics and Statistics},
		2:23--28, 2016.
		
		\bibitem{BergKahn2001}
		Jacob van~den Berg and Jeff Kahn.
		\newblock A correlation inequality for connection events in percolation.
		\newblock {\em Annals of probability} 29(1):123--126, 2001.
		
		\bibitem{HintLamm2019}
		Peter van Hintum and Piet Lammers.
		\newblock The bunkbed conjecture on the complete graph.
		\newblock {\em European Journal of Combinatorics}, 76:175–177, February 2019.	
	\end{thebibliography}
\end{document}